\documentclass[a4paper, DIV=12, 11pt]{amsart}
\usepackage[utf8]{inputenc}
\usepackage{amssymb, amsthm, amsmath, thmtools, mathtools}
\usepackage{amsfonts}
\usepackage[abbrev]{amsrefs} 

\usepackage{graphicx}
\usepackage{hyperref}
\usepackage[bottom, marginal]{footmisc}
\usepackage{todonotes}

\declaretheoremstyle[headfont=\normalfont]{normalhead}
\newtheorem{lemma}{Lemma}[section]
\newtheorem{theorem}[lemma]{Theorem}
\newtheorem{proposition}[lemma]{Proposition}
\newtheorem{corollary}[lemma]{Corollary}
\newtheorem{definition}[lemma]{Definition}
\newtheorem{remark}[lemma]{Remark}

\newtheorem*{acknowledgement}{Acknowledgement}
\newtheorem*{AIuse}{AI Usage Disclosure}

\newcounter{mt}

\newtheorem{maintheorem}[mt]{Theorem}

\newcommand{\R}{\mathbb{R}}

\DeclareMathOperator{\vol}{vol}

\DeclareMathOperator{\supp}{supp}

\DeclareMathOperator{\GL}{GL}

\DeclareMathOperator{\Aff}{\mathrm{Aff}}

\DeclareMathOperator{\SO}{\mathrm{SO}}

\newcommand{\calP}{\mathcal{P}}

\renewcommand{\H}{\mathbb{H}}
\renewcommand{\S}{\mathbb{S}}
\DeclareMathOperator{\CV}{CV}

\DeclareMathOperator{\PGL}{PGL}
\DeclareMathOperator{\Ad}{Ad}

\DeclareMathOperator{\loc}{loc}
\DeclareMathOperator{\Id}{Id}

\numberwithin{equation}{section}

\author{Jonas Knoerr}
\title{Hadwiger's Theorem on hyperbolic space}
\date{}

\newcommand{\Addresses}{{
		\bigskip
		\footnotesize
		
		Jonas Knoerr, \textsc{Institute of Discrete Mathematics and Geometry, TU Wien, Wiedner Hauptstrasse 8-10, 1040 Wien, Austria}\par\nopagebreak
		\textit{E-mail address}: \texttt{jonas.knoerr@tuwien.ac.at}
		
		\medskip
	}}
	
\makeatletter
\def\blfootnote{\xdef\@thefnmark{}\@footnotetext}
\makeatother

\makeindex
\begin{document}
\maketitle
\begin{abstract}
	We present a proof classifying continuous isometry invariant valuations on polytopes in hyperbolic space $\mathbb{H}^n$ as linear combinations of the intrinsic volumes. The argument relies on an adaptation of the spherical case recently established by the author that circumvents an approximation step which does not generalize to the hyperbolic setting.
\end{abstract}
\blfootnote{2020 \emph{Mathematics Subject Classification}. 52B45, 52A55, 52B11.\\
	\emph{Key words and phrases}. hyperbolic space, polytopes, valuation, intrinsic volumes.\\}
\tableofcontents

\section{Introduction}
	
	The most famous classical result in geometric valuation theory is Hadwiger's characterization of the intrinsic volumes.
		\begin{theorem}[\cite{HadwigerVorlesungenuberInhalt1957}]
		\label{theorem:Hadwiger}
		Let $\mu:\mathcal{K}(\R^n)\rightarrow\R$ be a continuous, translation and $\SO(n)$-invariant valuation. Then $\mu$ is a linear combination of the intrinsic volumes $V_0,\dots,V_n$.
	\end{theorem}
	Here, $\mathcal{K}(\R^n)$ denotes the space of convex bodies in $\R^n$, i.e. the set of all nonempty compact and convex subsets of $\R^n$ equipped with the Hausdorff metric, and a functional $\mu:\mathcal{S}\rightarrow\R$ defined on some family of sets $\mathcal{S}$ is called a valuation if
	\begin{align*}
		\mu(K)+\mu(L)=\mu(K\cup L)+\mu(K\cap L)
	\end{align*}
	for all $K,L\in\mathcal{S}$ such that $K\cup L,K\cap L\in\mathcal{S}$. In the last 25 years, similar characterization results have been obtained for continuous translation invariant valuations on convex bodies that are invariant under other compact subgroups of the general linear group $\GL(n,\R)$, compare \cite{BernigHadwigertypetheorem2009,BernigInvariantvaluationsquaternionic2012,BernigIntegralgeometry$G_2$2011,BernigSolanesKinematicformulasquaternionic2017,AleskerHardLefschetztheorem2003,KotrbatyWannererIntegralgeometryoctonionic2025}. Most of these classification results rely on the foundational work by Alesker \cite{AleskerDescriptiontranslationinvariant2001}, which has led to the discovery of several new structures on the space of continuous and translation invariant valuations on convex bodies, with a variety of applications in convex, differential, and integral geometry \cite{AleskerFaifmanConvexvaluationsinvariant2014,AleskerDescriptioncontinuousisometry1999,BernigEtAlHardLefschetztheorem2024,FuStructureunitaryvaluation2006,BernigFuHermitianintegralgeometry2011,BernigHugKinematicformulastensor2018,WannererIntegralgeometryunitary2014,Wannerermoduleunitarilyinvariant2014}. \\
	It has been a long-standing conjecture whether a version of Hadwiger's characterization holds for other real space forms, i.e. for the Euclidean sphere $\S^n$ and hyperbolic space $\H^n$. The corresponding results are easy to show for $\S^1$ and $\H^1$, and the two-dimensional cases were considered by Klain \cite{KlainIsometryinvariantvaluations2006} for the hyperbolic plane and Klain--Rota \cite{KlainRotaIntroductiongeometricprobability1997}*{Theorem 11.3.1} for the two-dimensional sphere. Recently, the author established the conjectured spherical Hadwiger Theorem for continuous isometry invariant valuations on spherical polytopes \cite{KnoerrIsometryinvariantvaluations2026}, which is based on a generalization of Hadwiger's Theorem to measurable valuations on polytopes in $\R^n$ in \cite{KnoerrRigidmotioninvariant2026}. In particular, the classification holds for spherically convex bodies. In this article, we present the proof of the corresponding result for hyperbolic space.\\
	
	Let us introduce some notation. We consider $\R^{n,1}=\R^{n+1}$ with the standard Lorentzian metric $\langle x,y\rangle=- x_0y_0+\sum_{i=1}^n x_iy_i$ and the hyperbolic space $\H^n$ in the hyperboloid model, i.e.
	\begin{align*}
		\H^n=\{x\in \R^{n,1}:\langle x,x\rangle =-1, x_0>0\}
	\end{align*}
	with Riemannian metric given by the restriction of $\langle\cdot,\cdot\rangle$ to the tangent spaces. We let $\calP(\H^n)$ denote the space of all polytopes in $\H^n$ equipped with the hyperbolic Hausdorff metric, where a polytope is by definition the geodesic convex hull of a finite set of points. Let $\SO^+(1,n)$ denote the identity component of the isometry group of $\H^n$. The main result of this article is the following extension of Hadwiger's Theorem to hyperbolic space.	
	\begin{maintheorem}\label{maintheorem:hyperbolicHadwiger}
		Let $\mu:\calP(\H^n)\rightarrow\R$ be a continuous and $\SO^+(1,n)$-invariant valuation. Then $\mu$ is a linear combination of the hyperbolic intrinsic volumes $V_0^h, \dots,V_n^h$.
	\end{maintheorem}
	Note in particular that this result implies the corresponding classification of continuous and $\SO^+(1,n)$-invariant valuations on geodesically convex compact sets in $\H^n$, since any such set can be approximated in the hyperbolic Hausdorff metric by polytopes.\\
	As in all proofs of Hadwiger's Theorem (see \cite{HadwigerVorlesungenuberInhalt1957,Chensimplifiedelementaryproof2004,KlainshortproofHadwigers1995,KnoerrRigidmotioninvariant2026}) and the spherical Hadwiger Theorem in \cite{KnoerrIsometryinvariantvaluations2026}, the argument can be reduced to a characterization of simple valuations, where we call a valuation on $\calP(\H^n)$ simple if it vanishes on lower dimensional polytopes.
	\begin{maintheorem}
		\label{maintheorem:SimpleValuation}
		Let $\mu:\calP(\H^n)\rightarrow\R$ be a continuous and $\SO^+(1,n)$-invariant valuation that is in addition simple. Then $\mu$ is a multiple of the hyperbolic volume.
	\end{maintheorem}
	The proof we will discuss in this article is an extension of the approach used in \cite{KnoerrIsometryinvariantvaluations2026} for the spherical case. The main idea is to exploit a differentiability property satisfied by valuations on polytopes in $\R^n$ that are affine smooth (see \autoref{section:affineSmooth} for the precise definition), i.e. that satisfy a certain smoothness property with respect to the natural action of the affine group on these valuations. This property can be used to transfer several constructions by Alesker \cite{AleskerTheoryvaluationsmanifolds.2006} for so-called quasi-smooth valuations to valuations on polytopes. The main idea behind these constructions is that they allow for the interpretation of suitable derivatives of the valuations as \emph{translation invariant} valuations on the tangent space at a given point. These derivatives are not necessarily continuous valuations, however, they are always measurable. Since the results in \cite{KnoerrRigidmotioninvariant2026} extend \autoref{theorem:Hadwiger} to measurable valuations on polytopes, this approach provides a way to obtain local information about the valuations under suitable invariance and regularity assumptions. In particular, it allows for a characterization of smooth measures in terms of the vanishing properties of these derivatives.\\
	For the spherical case, this is sufficient to obtain the desired characterization result. More precisely, the action of $\GL(n+1,\R)$ on lines in $\R^{n+1}$ gives rise to a natural action of $\GL(n+1,\R)$ on continuous valuations on spherical polytopes. Since the isometry group of $\S^n$ is compact, a simple mollification procedure shows that any such valuation can be approximated by isometry invariant valuations that restrict to affine smooth valuations in coordinate charts corresponding to the gnomonic projection. Since these can be investigated through their derivatives, this provides the necessary simplification to obtain the general classification.\\
	Most of these arguments generalize directly to hyperbolic space: There is a natural partial action of the projective linear group $\PGL(n+1,\R)$ on $\H^n$, which gives rise to a partial action on $\calP(\H^n)$, and thus to a notion of smoothness for valuations on $\calP(\H^n)$. Under the gnomonic projection, i.e. in the Beltrami--Klein model, this action corresponds to a partial action of the affine group on polytopes in the unit ball. Since the constructions from \cite{KnoerrIsometryinvariantvaluations2026} are local in nature, this is in principle sufficient to obtain \autoref{maintheorem:hyperbolicHadwiger} for valuations that satisfy additional smoothness assumptions with respect to the partial action of $\PGL(n+1,\R)$. However, in contrast to the spherical case, there is no direct way to obtain the general result using a direct approximation argument: Since we only have a partial action of $\PGL(n+1,\R)$ on $\calP(\H^n)$, the standard mollification procedure with suitable compactly supported smooth functions on a neighborhood of the identity in $\PGL(n+1,\R)$ only produces valuations on subsets of $\calP(\H^n)$. The more restrictive problem, however, is the non-compactness of the group $\SO^+(1,n)$, which prevents any simple approximation argument. In other words, there is no direct way to establish a priori that any continuous $\SO^+(1,n)$-invariant valuation can be approximated by invariant valuations with the required regularity properties. This is not a problem in the spherical setting, since we can average any approximating sequence with respect to the Haar measure on the corresponding compact isometry group.\\
	
	The solution to this problem comes from a more careful examination of the properties of the locally defined valuations obtained from the mollification. More precisely, after fixing the point $e_0\in \H^n\subset\R^{n,1}$, the elements in $U_0\subset \PGL(n+1,\R)$ that map $e_0$ to a point in $\H^n$ may be decomposed into elements of the stabilizer $\mathcal{S}$ of $e_0$ in $\PGL(n+1,\R)$ and elements in $\SO^+(1,n)$. Since the locally defined mollified valuations are obtained by convolution with respect to the Haar measure on $\PGL(n+1,\R)$, they naturally inherit various compatibility properties with respect to the action of $\mathcal{S}$ and $\SO^+(1,n)$. In particular, these properties descend to the derivatives of the corresponding affine smooth valuations in the Beltrami--Klein model at the origin. This allows us to examine the relevant derivatives of all possible local mollifications in the same framework, which we use to show that the necessary vanishing properties hold for simple $\SO^+(1,n)$-invariant valuations independent of the chosen mollifier, compare \autoref{theorem:SimpleLocallyW_n}. This reduces the problem to characterizing a certain distribution on $C^\infty_c(U_0)$, which turns out to be invariant under a transitive group action and therefore a multiple of the unique invariant measure.
	
	\begin{acknowledgement}
		The initial proof was generated by Claude Fable 5 in response to prompts by Florian Besau on possible generalizations of the approach to the spherical case in \cite{KnoerrIsometryinvariantvaluations2026} to hyperbolic space, and I thank him for providing me with the output. \\
		
		This research was funded in whole or in part by the Austrian Science Fund (FWF), \href{https://www.doi.org/10.55776/PAT4205224}{10.55776/PAT4205224}. 
	\end{acknowledgement}
	
	\begin{AIuse}
		Claude Fable 5 produced the initial proof. The author verified and revised the argument, and rewrote and restructured the proof to make it more accessible. 
	\end{AIuse}

\section{Affine smooth valuations}\label{section:affineSmooth}
	We refer to \cite{SchneiderConvexbodiesBrunn2014}*{Section 6} and \cite{KlainRotaIntroductiongeometricprobability1997} for a general background on convex geometry and valuations on polytopes. In this section, we review some of the key results from \cite{KnoerrIsometryinvariantvaluations2026}. For an open and convex subset $\Omega\subset \R^n$, we denote the set of polytopes contained in $\Omega$ by $\calP(\Omega)$, which we equip with the Hausdorff metric. Let us call a valuation $\mu:\calP(\Omega)\rightarrow\R$ affine smooth if for every $P\in\calP(\R^n)$, the map
	\begin{align*}
		\{A\in \Aff(n,\R):A(P)\subset \Omega\}&\rightarrow\R\\
		A&\mapsto \mu(A(P))
	\end{align*}
	is smooth. The following result was shown in \cite{KnoerrIsometryinvariantvaluations2026} for the case $\Omega=\R^n$. The proof holds verbatim in the more general case.
	\begin{proposition}[\cite{KnoerrIsometryinvariantvaluations2026}*{Theorem 3.1}]\label{prop:localAffineSoothness}
		Let $\Omega\subset\R^n$ be open and convex, $\mu:\calP(\Omega)\to\R$ an affine smooth valuation. Then for every $P\in\calP(\R^n)$ the map 
		\begin{align*}
			\{(t,g,x)\in [0,\infty)\times \GL(n,\R)\times\R^n: g[tP+x]\subset\Omega\}&\rightarrow\R\\
			(t,g,x)&\mapsto\mu(g[tP+x])
		\end{align*}
		is smooth.
	\end{proposition}	
	Let $\CV(\Omega)$ denote the space of continuous valuations $\mu:\calP(\Omega)\rightarrow\R$, $\CV(\Omega)^{sm}$ the subspace of affine smooth valuations. For $x\in \Omega$ and $i\in\mathbb{N}_0$, we define the maps
	\begin{align*}
		D^i_x\mu:\calP(\R^n)&\rightarrow\R\\
		D^i_x\mu[P]&:=\frac{d^i}{dt^i}\Big|_0\mu(tP+x).
	\end{align*}
	Consider the decreasing filtration on $\CV(\Omega)^{sm}$ given by
	\begin{align*}
		W_k(\Omega):=\{\mu\in\CV(\Omega)^{sm}: D^i_x(\mu)[P]=0~ \forall P\in\calP(\R^n), x\in\Omega, 0\le i<k\}.
	\end{align*}
	The next result follows with the identical arguments as in \cite{KnoerrIsometryinvariantvaluations2026}.
	
	\begin{proposition}[\cite{KnoerrIsometryinvariantvaluations2026}*{Prop.~3.5, 3.7, 3.8}]\label{prop:PropertiesFiltration}
	Let $\mu\in\CV(\Omega)^{sm}$. Then the following holds.
	\begin{enumerate}
		\item If $\mu\in W_k(\Omega)$ and $x\in\Omega$, then $D^k_x\mu:\calP(\R^n)\to\R$ is a translation invariant, measurable, and $k$-homogeneous valuation.
		\item If $\mu$ is simple, so is $D^k_x\mu$.
		\item $W_{n+1}(\Omega)=0$.
		\item If $\mu\in W_n(\Omega)$, then $\mu$ is a smooth measure. More precisely, $\mu(P)=\frac1{n!}\int_P c(x)\,dx$ for the smooth function on $\Omega$ given by $c(x):=D^n_x\mu[P_0]$, where $P_0\in\calP(\R ^n)$ is any polytope with $\vol_n(P_0)=1$.
	\end{enumerate}
\end{proposition}

\begin{remark}
	The previous results heavily rely on the fact that valuations on polytopes satisfy the full Inclusion-Exclusion Principle, compare \cite{SchneiderConvexbodiesBrunn2014}*{Theorem~6.2.3}. The standard proof given in \cite{SchneiderConvexbodiesBrunn2014} is only concerned with the case $\Omega=\R^n$, however, it is easy to check that the proof holds verbatim for valuations on polytopes contained in a fixed open convex set.
\end{remark}

\section{Hyperbolic geometry}
	
\subsection{Gnomonic projection and the Beltrami--Klein model}
	Recall that we consider $\R^{n,1}=\R^{n+1}$ with the standard Lorentzian metric $\langle x,y\rangle=- x_0y_0+\sum_{i=1}^n x_iy_i$ and the hyperbolic space $\H^n$ in the hyperboloid model $\H^n=\{x\in \R^{n,1}:\langle x,x\rangle =-1, x_0>0\}$. In order to relate valuations on $\calP(\H^n)$ to affine smooth valuations on polytopes in $\R^n$, it will be convenient to consider for $u\in \H^n$ the gnomonic projection
		\begin{align*}
			G_u:\H^n&\rightarrow u^\perp=\{x\in \R^{n,1}:\langle x,u\rangle=0\}\\
			G_u(v)&=\frac{v-\langle u,v\rangle u}{-\langle u,v\rangle}.
		\end{align*}
		Its  image is $\Omega(u):=\{x\in u^\perp:\langle x,x\rangle <1\}$, which is the open ball with radius $1$ centered at the origin, and $G_u:\H^n\rightarrow \Omega(u)$ is a diffeomorphism with inverse
		\begin{align*}
			G_u^{-1}(x)=\frac{u+x}{\sqrt{1-\langle x,x\rangle}}.
		\end{align*}
		As in the spherical case, the gnomonic projection maps $v\in \H^n$ to the intersection point of the associated line with the hyperplane $u+u^\perp$. The open ball $\Omega(u)$ equipped with the induced Riemannian metric is called the Beltrami--Klein model of hyperbolic space. It has the important property that it is a projective model, i.e. geodesics in $\H^n$ are mapped to lines in $\Omega(u)$. In particular, polytopes in $\calP(\H^n)$ are in $1$-to-$1$-correspondence with polytopes in $\calP(\Omega(u))$. The following is a simple consequence of this observation.
		\begin{lemma}\label{lemma:Gnomonic}
			For every $u\in\H^n$, the map $G_u$ induces a homeomorphism between $\calP(\H^n)$ and  $\calP(\Omega(u))\subset \calP(u^\perp)$. In addition, it maps lower dimensional polytopes to lower dimensional polytopes and simplices to simplices.
		\end{lemma}
		
		\subsection{Partial action of $\PGL(n+1,\R)$ on $\H^n$ and $\calP(\H^n)$}
		
		We define a partial action of $\GL(n+1,\R)$ on $\H^n$ in the following way: If $x\in \H^n$ and $g\in\GL(n+1,\R)$ are such that $\langle gx,gx\rangle<0$ and $\langle gx,e_0\rangle<0$, then we set
		\begin{align*}
			g(x):=\frac{gx}{\sqrt{-\langle gx,gx\rangle}}.
		\end{align*}
		Note that positive multiples of the identity act trivially under this partial action, i.e. it factors through the projective general linear group $\PGL(n+1,\R)$. In fact, there is a natural interpretation of this action in terms the projectivization $\mathbb{P}(\R^{n,1})$ of $\R^{n,1}$: Note that the map $\H^n\rightarrow \mathbb{P}(\R^{n,1})$ that maps $u\in\H^n$ to the line $[u]\in \mathbb{P}(\R^{n,1})$ is a diffeomorphism onto its image. Under this identification, this action corresponds to the partial action of $\PGL(n+1,\R)$ on time-like lines in $\mathbb{P}(\R^{n,1})$, i.e. lines $[u]$ such that $\langle u,u\rangle<0$. \\	
		
		Our constructions will revolve around the fixed point $e_0\in \H^n$ and the action of the stabilizer of this point on the relevant valuations. Let $\mathcal{S}\subset \PGL(n+1,\R)$ be the stabilizer of the ray through $e_0\in \R^{n+1}$ in $\PGL(n+1,\R)$, which may be identified with the subgroup of $\GL(n+1,\R)$ consisting of the matrices
		\begin{align*}
			g_{x,A}:=\begin{pmatrix}
				1 & x^T\\
				0 & A
			\end{pmatrix}, x\in e_0^\perp, A\in\GL(e_0^\perp).
		\end{align*}
		This subgroup is isomorphic to the semi-direct product $\mathcal{N}\rtimes \mathcal{A}$, where
		\begin{align}
			\label{eq:DefFactorsStabilizer}
			\begin{split}
				\mathcal{N}=&\{n_x:=g_{x,\Id }: x\in e_0^\perp\}\cong e_0^\perp,\\ \mathcal{A}=&\{a_L:=g_{0,L}:L\in\GL(e_0^\perp)\}\cong \GL(n,\R),
			\end{split}
		\end{align}
		and $\mathcal{A}$ acts on $\mathcal{N}$ by
		\begin{align*}
			a_Ln_xa_L^{-1}=n_{L^{-T}x},
		\end{align*}
		for $ L\in\GL(e_0^\perp)$, $x\in e_0^\perp$.\\
		
		Note that the stabilizer of $e_0$ in $\SO^+(1,n)$ is $K:=\mathrm{Stab}_{\SO^+(1,n)}(e_0)\cong \SO(n)$ acting on $e_0^\perp=\mathrm{span}(e_1,\dots,e_n)\cong \R^n$. The last isomorphism uses that  $\langle\cdot,\cdot\rangle$ restricts to the standard inner product on $e_0^\perp\cong \R^n$.\\
		With respect to the point $e_0$, the partial action of $\PGL(n+1,\R)$ may be decomposed in the following way with respect to the stabilizer $\mathcal{S}$ of $e_0$ in $\PGL(n+1,\R)$ and the group $\SO^+(1,n)$, which operates transitively on $\H^n$.
		\begin{lemma}\label{lemma:partialActionHomSpace}
			Consider the set
			\begin{align*}
				U_0:=\{g\in \PGL(n+1,\R): g([e_0])~\text{defines a time-like line in}~\R^{n,1} \}.
			\end{align*}
			Then the following holds:
			\begin{enumerate}
				\item $\mathcal{S}\cap \SO^+(1,n)=K$.
				\item $U_0=\mathcal{S}\cdot \SO^+(1,n)$.
				\item The following map is a surjective submersion: 
					\begin{align*}
						H:=\mathcal{S}\times \SO^+(1,n)&\rightarrow U_0\\
						(\sigma, h)&\mapsto \sigma\cdot h^{-1}.
					\end{align*}
				\item The restriction of the Haar measure of $\PGL(n+1,\R)$ to $U_0$ is left $\mathcal{S}$-invariant and right $\SO^+(1,n)$-invariant.
			\end{enumerate}
		\end{lemma}
		\begin{proof}
			The only nontrivial statement is the invariance of the Haar measure, however, this is a simple consequence of the fact that $\PGL(n+1,\R)$ is a unimodular group, i.e. the Haar measure is both left and right invariant.
		\end{proof}
		Let us consider $H$ as a Lie group with componentwise multiplication. The previous result has the following direct consequence.
		\begin{corollary}
			Let $K_\Delta$ be the subgroup of $H$ given by the diagonal embedding of $K$ into $H$, i.e. $K_\Delta=\{(k,k)\in\mathcal{S}\times \SO^+(1,n): k\in K\}$. Then $U_0\cong H/K_{\Delta}$ as a homogeneous space of $H$ with respect to the action $(\sigma,h)\cdot g:=\sigma \cdot g\cdot h^{-1}$.
		\end{corollary}
		Consider the subset 
		\begin{align*}
			\mathcal{D}=\{(g,P)\in\PGL(n+1,\R)\times \calP(\H^n): g(P)\subset \H^n\}.
		\end{align*}
		This is an open subset of $\PGL(n+1,\R)\times\calP(\H^n)$ and we obtain a continuous partially defined action
		\begin{align*}
			\mathcal{D}&\rightarrow\calP(\H^n)\\
			(g,P)&\mapsto g(P).
		\end{align*}
			Next, we are going to consider the interaction of the partial action of $\PGL(n+1,\R)$ on $\H^n$ and the gnomonic projection $G_{e_0}$ at the point $e_0$. Using the notation in Eq.~\eqref{eq:DefFactorsStabilizer}, we define for $t>0$, the elements $g_t:=a_{t\Id_{e_0^\perp}}\in \mathcal{A}$ and consider the maps $\tau_x:=n_x^T$, $x\in e_0^\perp$ in $\PGL(n+1,\R)$, i.e. $\tau_x$ is induced by the matrix in $\GL(n+1,\R)$ given by
		\begin{align*}
			\begin{pmatrix}
				1 &0\\
				x & \mathrm{Id_n}
			\end{pmatrix}.
		\end{align*}

		We note the following relation between the partial action of the affine group on $\calP(\Omega(e_0))$ and this partial action.
		\begin{corollary}\label{corollary:equivariancePartialActions}
			Let $P\in\calP(e_0^\perp)$ and $L\in\GL(e_0^\perp)$, $x\in \Omega(e_0)$ be such that $LP+x\subset \Omega(e_0)$. Then
			\begin{align*}
				G_{e_0}^{-1}(LP+x)=\tau_xa_L(G_{e_0}^{-1}(P)).
			\end{align*}
		\end{corollary}
		In other words, $\tau_x$ acts as a translation in the chart.
	
\section{Continuous valuations on $\calP(\H^n)$}
	\subsection{Local smoothing}
		\label{section:localSmoothing}
		The set $U_0\subset\PGL(n+1,\R)$ considered in \autoref{lemma:partialActionHomSpace} is obviously an open neighborhood of the identity. We will use this set to construct locally defined mollifications of elements in $\CV(\H^n)$. We require the following notation: For $b\in \PGL(n+1,\R)$ and a function $\psi\in C^\infty_c(U_0)$, we set
		\begin{align*}
			\psi^b(g):=\psi(bg),\quad g\in U_0,
		\end{align*}
		whenever this is well defined. Since $U_0$ is a homogeneous space with respect to the action of $H=\mathcal{S}\times \SO^+(1,n)$ considered in \autoref{lemma:partialActionHomSpace}, we obtain a well defined and continuous action on $C^\infty_c(U_0)$ given by
		\begin{align*}
			[\pi(\sigma,h)\psi](g):=\psi(\sigma^{-1}\cdot g\cdot h),\quad (\sigma,h)\in H,~g\in U_0.
		\end{align*}

		\begin{definition}
			Let $A\subset U_0$ be a compact subset and consider the sets 
			\begin{align*}
				\calP_A(\H^n):=&\{P\in\calP(\H^n): (g^{-1},P)\in\mathcal{D}~\text{for all}~g\in A\}\\
				C_A^\infty(U_0):=&\{\psi\in C^\infty_c(U_0):\supp \psi\subset A\}.
			\end{align*}
			For $\mu\in \CV(\H^n)$ and $\psi\in C_c(U_0)$, we define $\mu_{\psi}:\calP_{\supp\psi}(\H^n)\rightarrow\R$ by
				\begin{align*}
				\mu_\psi(P):=\int_{\PGL(n+1,\R)} \psi(g)\mu(g^{-1}(P))d\lambda_{\PGL(n+1,\R)}.
			\end{align*}
		\end{definition}
		Here, we denote by $\lambda_G$ a Haar measure on a Lie group $G$. Note that the integral is well defined due to the restrictions imposed on the polytopes.

		\begin{lemma}\label{lemma:propertiesLocalSmoothing}
			Let $\mu\in \CV(\H^n)$ and $\psi\in C^\infty_c(U_0)$.
			\begin{enumerate}
				\item If $\mu$ is $\SO^+(1,n)$-invariant, then  $\mu_{R_h\psi}=\mu_{\psi}$ for all $h\in\SO^+(1,n)$, where $R_h\psi(g):=\psi(gh)$ for $g\in U_0$.
				\item For every $g\in \PGL(n+1,\R)$ such that $\psi^g\in C^\infty_c(U_0)$, we have $\mu_{\psi^g}(P)=\mu_\psi(gP)$ for all $P\in\calP(\H^n)$ such that both sides are defined.
				\item If $\mu$ is simple, then $\mu_\psi$ vanishes on lower dimensional polytopes in $\calP_{\supp\psi}(\H^n)$.
				\item For any fixed $P\in\calP(\H^n)$, the map defined on the open set
				\begin{align*}
					\{g\in \PGL(n+1,\R):g(P)\in \calP_{\supp\psi}(\H^n)\}&\rightarrow\R\\
					g&\mapsto \mu_\psi(gP)
				\end{align*}
				is smooth.
				\item If $(\psi_j)_j$ is a smooth approximation of the identity in $C^\infty_c(U_0)$ with supports shrinking to $\{\Id\}$, then 
				\begin{align*}
					\lim_{j\rightarrow\infty} \mu_{\psi_j}(P)=\mu(P)
				\end{align*}
				pointwise on $\calP(\H^n)$.
			\end{enumerate}
		\end{lemma}
		\begin{proof}
			The first three claims follow directly from the definition and the invariance of the Haar measure, while the last two claims are standard properties of convolutions of smooth and continuous functions on Lie groups. For the last step, this uses that any $P\in\calP(\H^n)$ belongs to $\calP_{\supp_{\psi_j}}(\H^n)$ for all $j$ large enough, since the supports shrink to the identity, so the expression in the limit is well defined for $j$ large enough.
		\end{proof}
		
		For a relatively compact open convex subset $\Omega\subset \Omega(e_0)\subset e_0^\perp$, we may consider the family of polytopes $G_{e_0}^{-1}(\calP(\Omega))\subset \calP(\H^n)$. Since this family is contained in a bounded subset of $\H^n$, there exists a compact neighborhood $A$ of the identity in $U_0$ such that $G_{e_0}^{-1}(\calP(\Omega))\subset\calP_{A}(\H^n)$. Conversely, if $A\subset U_0$ is compact, then there is an open and convex neighborhood $\Omega\subset \Omega(e_0)$ of the origin such that $G^{-1}_{e_0}(\calP(\Omega))\subset \calP_A(\H^n)$.
		
		\begin{corollary}\label{corollary:localAffineSmoothness}
			Let $\mu\in \CV(\H^n)$ and $\Omega\subset \Omega(e_0)$ be a relatively compact open convex subset. For every $\psi\in C^\infty_c(U_0)$ such that $G_{e_0}^{-1}(\calP(\Omega))\subset \calP_{\supp\psi}(\H^n)$, the map $G_{e_0*}\mu_\psi:\calP(\Omega)\rightarrow\R$ defined by
			\begin{align*}
				G_{e_0*}\mu_\psi (P):=\mu_{\psi}(G_{e_0}^{-1}(P))
			\end{align*}
			defines an affine smooth valuation in $\CV(\Omega)$.
		\end{corollary}
		\begin{proof}
			By construction, the map is well defined. Moreover, it follows directly from \autoref{lemma:Gnomonic} that this defines a continuous valuation. The fact that $G_{e_0*}\mu_\psi$ is affine smooth follows from \autoref{lemma:propertiesLocalSmoothing} in combination with the equivariance properties of the gnomonic projection in \autoref{corollary:equivariancePartialActions}.
		\end{proof}

	\subsection{Definition and properties of $\Lambda^\mu_k$}
		Let $\mu\in\CV(\H^n)$. For $\psi\in C^\infty_c(U_0)$, we define
		\begin{align*}
			\Lambda^\mu_k(\psi):\calP(e_0^\perp)\rightarrow\R
		\end{align*}
		by
		\begin{align*}
			\Lambda^\mu_k(\psi)[P]:=D^k_0(G_{e_0*}\mu_\psi)[P]=s^{-k}\frac{d^k}{dt^k}\Big|_0\mu_{\psi}\left(g_tG_{e_0}^{-1}(sP)\right),
		\end{align*}
		where we may choose any $s>0$ such that $G_{e_0}^{-1}(sP)\in \calP_{\supp\psi}(\H^n)$ due to the chain rule. Note that this is well-defined due to \autoref{corollary:localAffineSmoothness} and \autoref{prop:localAffineSoothness}, and that we used the equivariance properties of the gnomonic projection from \autoref{corollary:equivariancePartialActions} in the last step, where $g_t=a_{t\Id _{e_0^\perp}}$ by definition. 
		
		While the filtration in \autoref{section:affineSmooth} relies on the properties of $D^k_x(G_{e_0*}\mu_\psi)$ for arbitrary $x$ in the relevant domain, we may use the additional dependence on $\psi\in C^\infty_c(U_0)$ to deal with these translations. More precisely, the following holds.
		\begin{lemma}\label{lemma:TranslationMaps}
			Let $\psi\in C^\infty_c(U_0)$. If $x\in \Omega(e_0)$ is such that $\psi^{\tau_x}\in C^\infty_c(U_0)$, then 
			\begin{align*}
				D_x^k(G_{e_0*}\mu_\psi)=\Lambda^\mu_k(\psi^{\tau_x}).
			\end{align*}
			In particular, there exists a neighborhood $V_{\supp\psi}\subset e_0^\perp$ of $0$ depending on $\supp\psi$ only such that this relation holds for all $x\in V_{\supp\psi}$.
		\end{lemma}
		\begin{proof}
			Due to \autoref{corollary:equivariancePartialActions}, we have
			\begin{align*}
				D_x^k(G_{e_0*}\mu_\psi)[P]=&\frac{d^k}{dt^k}\Big|_0\mu_{\psi}(G_{e_0}^{-1}(tP+x))=\frac{d^k}{dt^k}\Big|_0\mu_{\psi}(\tau_x g_tG_{e_0}^{-1}(P))\\
				=&\frac{d^k}{dt^k}\Big|_0\mu_{\psi^{\tau_x}}( g_tG_{e_0}^{-1}(P)),
			\end{align*}
			where we used \autoref{lemma:propertiesLocalSmoothing} (2) in the last step. This shows the first claim. For the second claim, observe that $\psi^{\tau_x}\in C_c(U_0)$ for all $x$ in a small neighborhood of $0$.
		\end{proof}
		
		Since it is slightly cumbersome to keep track of the domains of these locally defined valuations, we use the previous result to introduce a more easily to handle variant of the filtration from \autoref{section:affineSmooth}. 
		\begin{definition}
			Let $\mu\in \CV(\H^n)$. We will say that $\mu$ belongs to $W_k^{\loc}$ if $\Lambda^\mu_i$ vanishes identically for all $0\le i<k$.
		\end{definition}
		The benefit of this notion is that we are just considering the countable family of functionals $\Lambda^\mu_i:C^\infty_c(U_0)\times \calP(e_0^\perp)\rightarrow\R$. In fact, all induction procedures will reduce to the finite range $0\le i\le n$. The relation to the filtration considered in \autoref{section:affineSmooth} is given by the following result.
		\begin{corollary}\label{corollary:EquivalenceLocalW_k}
			If $\mu\in \CV(\H^n)$ belongs to $W_k^{\loc}$, then for every relatively compact open and convex subset $\Omega\subset \Omega(e_0)$, there is a compact neighborhood $A_{\Omega}\subset U_0$ of the identity such that for every $\psi\in C^\infty_{A_\Omega}(U_0)$
			\begin{enumerate}
				\item $G_{e_0*}\mu_{\psi}\in \CV(\Omega)^{sm}$, and
				\item $G_{e_0*}\mu_{\psi}\in W_k(\Omega)$.
			\end{enumerate}
			Conversely, any $\mu$ satisfying this condition belongs to $W^{\loc}_k$. 
		\end{corollary}
		\begin{proof}
			Note first that if $\mu$ satisfies this condition, then $\Lambda^\mu_i=0$ for $0\le i<k$, since we may calculate the derivative on polytopes contained in any open neighborhood of $0\in \Omega(e_0)$ on which $\mu_\psi$, $\psi\in C^\infty_c(U_0)$, is well defined. \\
			Now assume that $\Lambda^\mu_i=0$ for all $0\le i<k$. If $\Omega\subset \Omega(e_0)$ is a relatively compact open convex set, then we may choose a compact neighborhood $A\subset U_0$ of the identity with $G^{-1}_{e_0}(\calP(\Omega))\subset \calP_{A}(\H^n)$. Since $\Omega$ is relatively compact, we may shrink $A$ if necessary to obtain a compact neighborhood $A_\Omega$ of the identity in $U_0$ such that $\psi^{\tau_x}\in C^\infty_c(U_0)$ for all $x\in \Omega$ and $\psi\in C^\infty_{A_\Omega}(U_0)$. Note that this implies in particular that $G_{e_0*}\mu_{\psi}\in\CV(\Omega)^{sm}$ is well defined for $\psi\in C^\infty_{A_{\Omega}}(U_0)$, compare \autoref{corollary:localAffineSmoothness}.
			In order to check that $G_{e_0*}\mu_{\psi}$ belongs to $W_k(\Omega)$, we use \autoref{lemma:TranslationMaps} to obtain, for $0\le i\le k-1$ and every $x\in\Omega$,
			\begin{align*}
				D_x^i(G_{e_0*}\mu_\psi)=\Lambda^\mu_i(\psi^{\tau_x})=0,
			\end{align*}
			since $\mu$ belongs to $W^{\loc}_k$. Thus $G_{e_0*}\mu_{\psi}\in W_k(\Omega)$ for all $\psi\in C^\infty_{A_\Omega}(U_0)$.
		\end{proof} 
		\begin{remark}\label{remark:W1Loc}
			Note in particular that $\mu\in \CV(\H^n)$ belongs to $W^{\loc}_1$ if and only if $\mu$ vanishes on singletons.
		\end{remark}

		The following equivariance properties are easy to check using \autoref{lemma:propertiesLocalSmoothing}.
		\begin{lemma}\label{lemma:equivPropertiesLambda}
			Let $\mu\in\CV(\H^n)$, $\psi\in C^\infty_c(U_0)$.
			\begin{enumerate}
				\item If $\mu$ is $\SO^+(1,n)$-invariant, then $\Lambda^\mu_k(R_h\psi)=\Lambda^\mu_k(\psi)$ for all $h\in\SO^+(1,n)$.
				\item $\Lambda^\mu_k(\psi^{a_L})(P)=\Lambda^\mu_k(\psi)[LP]$ for all $L\in\GL(e_0^\perp)$.
				\item If $\psi$ is $\Ad(K)$-invariant, then $\Lambda^\mu_k(\psi)$ is $\SO(e_0^\perp)$-invariant.
			\end{enumerate}
		\end{lemma}
		
		We next consider the subgroup $\mathcal{N}$ of the stabilizer $\mathcal{S}$, which acts as fractional linear maps in the chart corresponding to the gnomonic projection $G_{e_0}$.
		\begin{proposition}\label{proposition:InvarianceDerivativesFractionalLinearPart}
			Let $\mu\in \CV(\H^n)$ belong to $W^{\loc}_k$. Then
			\begin{align*}
				\Lambda^\mu_k(\psi^{n_x})=\Lambda^\mu_k(\psi)
			\end{align*}
			for every $x\in e_0^\perp$ and all $\psi\in C^\infty_c(U_0)$.
		\end{proposition}
		\begin{proof}
			Let us first show the claim for full dimensional simplices $\Delta\in \calP(e_0^\perp)$. Since $\Lambda^\mu_k(\psi)$ is a translation invariant $k$-homogeneous valuation by \autoref{prop:localAffineSoothness}, we may without loss of generality assume that $\Delta\subset \Omega(e_0)$ has a vertex at the origin (i.e. after rescaling and translating $\Delta$). Let $p_0,\dots,p_n$ denote the vertices of $\Delta$, where $p_0=0$. Note that
			\begin{align*}
				n_xg_t=\begin{pmatrix}
					1 & x^T\\
					0 & \Id 
				\end{pmatrix}\begin{pmatrix}
				1 & 0\\
				0 & t\Id 
				\end{pmatrix}=\begin{pmatrix}
				1 & tx^T\\
				0 & t\Id 
				\end{pmatrix}=\begin{pmatrix}
				1 & 0\\
				0 & t\Id 
				\end{pmatrix}\begin{pmatrix}
				1 & tx^T\\
				0 & \Id 
				\end{pmatrix}=g_t n_{tx},
			\end{align*}
			and that the vertices $p_i(t)$ of the simplex $G_{e_0}(n_{tx}G_{e_0}^{-1}(\Delta))$ are given by
			\begin{align*}
				p_i(t)=\frac{p_i}{1+t\langle x,p_i\rangle}.
			\end{align*}
			In particular, this simplex always has the vertex $p_0(t)$ at the origin. In other words,
			\begin{align*}
				G_{e_0}(n_{tx}G_{e_0}^{-1}(\Delta))=A_t(\Delta)
			\end{align*}
			for the smooth curve $t\mapsto A_t\in\GL(e_0^\perp)$ defined by $A_t(p_i)=p_i(t)$ for $1\le i\le n$. Note that $A_0= \Id _{e_0^\perp}$. Since $\mu$ belongs to $W^{\loc}_k$, we have
			\begin{align*}
				\Lambda^\mu_k(\psi^{n_x})[\Delta]=\frac{d^k}{dt^k}\Big|_0 \mu_{\psi^{n_x}}(g_t G_{e_0}^{-1}(\Delta))=\frac{d^k}{dt^k}\Big|_0 \mu_{\psi}(n_xg_t G_{e_0}^{-1}(\Delta))
			\end{align*}
			due to \autoref{lemma:propertiesLocalSmoothing}. Our previous discussion thus implies
			\begin{align*}
					\Lambda^\mu_k(\psi^{n_x})[\Delta]=\frac{d^k}{dt^k}\Big|_0 G_{e_0*}\mu_{\psi}(t (A_t \Delta)).
			\end{align*}
			Since $(t,A)\mapsto \mu_{\psi}(t (A \Delta))$ defines a smooth function on a neighborhood of $(0,\Id_{e_0^\perp})\in [0,\infty)\times \GL(e_0^\perp)$ by \autoref{prop:localAffineSoothness}, we may repeatedly apply the chain rule to obtain
			\begin{align*}
				\Lambda^\mu_k(\psi^{n_x})[\Delta]=&\sum_{j=0}^k\binom{k}{j}\frac{\partial^{k-j}}{\partial s^{k-j}}\Big|_0\frac{\partial^j}{\partial t^j}\Big|_0 G_{e_0*}\mu_{\psi}(t (A_s \Delta))\\
				=&\frac{d^k}{dt^k}\Big|_0G_{e_0*}\mu_{\psi}(t (A_0 \Delta))=\Lambda^\mu_k(\psi)[\Delta],
			\end{align*}
			where we used that $\mu$ belongs to $W^{\loc}_k$, so all lower order partial derivatives with respect to $t$ vanish in $t=0$. This shows the claim for full dimensional simplices. If $\Delta$ is a lower dimensional simplex, then it is a face of a full dimensional simplex, and, after translating and rescaling, we may use this full dimensional simplex to construct a smooth curve $t\mapsto A_t$ in $\GL(e_0^\perp)$ similar to the previous case. In particular, we can repeat the argument for the lower dimensional simplex $\Delta$ using this map. Thus $\Lambda^\mu_k(\psi^{n_x})[\Delta]=\Lambda^\mu_k(\psi)[\Delta]$ for all simplices $\Delta\in \calP(e_0^\perp)$. Since both sides define valuations on $\calP(e_0^\perp)$, we may triangulate any given polytope and use the Inclusion-Exclusion Principle (compare \cite{SchneiderConvexbodiesBrunn2014}*{Theorem~6.2.3}) to see that both sides coincide on arbitrary polytopes in $\calP(e_0^\perp)$. The claim follows.
		\end{proof}
		\begin{remark}
			This argument is a rather elegant work-around for the lack of a proper chain rule (compare the discussion in \cite{KnoerrIsometryinvariantvaluations2026}*{Remark~4.5}). The key observation is that the fractional linear map induced by $n_x$ in the chart given by the gnomonic projection acts as an affine map on any given simplex. For valuations belonging to $W^{\loc}_k$, the derivatives of the corresponding curve $t\mapsto A_t$ in the affine group do not contribute to the final result, so we obtain the expected chain rule for arbitrary polytopes from the valuation property. In particular, the same argument can be used in the spherical setting of \cite{KnoerrIsometryinvariantvaluations2026} instead of the more complicated approach through Taylor estimates.
		\end{remark}
		
		We end the discussion of the general properties of the maps $\Lambda^\mu_k$ with the following continuity result.
		\begin{lemma}\label{lemma:LambdaContinuousFunctional}
		Let $\mu\in \CV(\H^n)$ belong to $W_k^{\loc}$ and let $A\subset U_0$ be a compact subset. Then for every $P\in\calP(e_0^\perp)$ the map
		\begin{align*}
			C^\infty_A(U_0)&\rightarrow\R\\
			\psi&\mapsto \Lambda^\mu_k(\psi)[P] 
		\end{align*}
		is linear and continuous in the natural Fr\'echet topology on $C^\infty_A(U_0)$.
	\end{lemma}
	\begin{proof}
		Fix $s>0$ with $g_tG_{e_0}^{-1}(sP)\in \calP_A(\H^n)$ for all $t\in (0,\epsilon)$ for some $\epsilon>0$. Since $\mu$ belongs to $W_k^{\loc}$, we have for $\psi\in C^\infty_A(U_0)$,
		\begin{align*}
			\Lambda^\mu_k(\psi)[P]=s^k\frac{d^k}{dt^k}\Big|_0\mu_{\psi}(g_tG_{e_0}^{-1}(sP))=k!s^k\lim_{m\rightarrow\infty} m^k\mu_{\psi}(g_{1/m}G^{-1}_{e_0}(sP))
		\end{align*}
		as all lower order derivatives vanish. Note that there is $M\in\mathbb{N}$ such that $g_{1/m}G^{-1}_{e_0}(sP)\in \calP_A(\H^n)$ for all $m\ge M$. Thus, for each fixed $m\ge M$, we may estimate 
		\begin{align*}
			|\mu_{\psi}(g_{1/m}G^{-1}_{e_0}(sP))|\le \lambda_{\PGL(n+1,\R)}(A)\sup_{g\in A}|\mu(g^{-1}g_{1/m} G_{e_0}^{-1}(sP))|\cdot \|\psi\|_\infty,
		\end{align*} 
		which is bounded for each $m\ge M$ due to the continuity of $\mu$ and the fact that $A$ is compact. Thus the linear functional $\psi\mapsto \Lambda^\mu_k(\psi)[P]$ is the pointwise limit of a sequence of continuous and linear functionals on the Fr\'echet space $C^\infty_A(U_0)$. The Banach--Steinhaus Theorem therefore implies that this map is also continuous.
	\end{proof}
	\begin{remark}\label{remark:orderDistribution}
		The Banach--Steinhaus Theorem is not really required for the result (although it simplifies the argument): A more careful look at the proof of \cite{KnoerrIsometryinvariantvaluations2026}*{Theorem 3.1} shows that $\Lambda^\mu_k(\psi)[P]$ is obtained by decomposing $P$ into simplices and differentiating the action of the affine group acting on the valuation on suitable polytopes derived from these simplices. In particular, we may use the invariance of the Haar measure to transfer these derivatives to the functions $\psi$. Since we are differentiating $k$-times, one easily checks that there exists a constant $C(A,P,\mu,k)>0$ such that for all $\psi\in C^\infty_A(U_0)$,
		\begin{align*}
			|\Lambda^\mu_k(\psi)[P]|\le C(A,P,\mu,k) \|\psi\|_{C^k}.
		\end{align*}
		In particular, for every $P\in \calP(e_0^\perp)$, we obtain a distribution of order at most $k$. Moreover, this shows that $\psi\mapsto \Lambda^\mu_k(\psi)[P]$ is continuous independent of the vanishing properties of $\Lambda^\mu_i$, $0\le i<k$. Since we do not need these more refined properties, we leave the necessary arguments to the reader.
	\end{remark}
		
	\subsection{Vanishing properties of $\Lambda^\mu_k$, $0\le k\le n-1$, in the invariant case}

		We begin this section with the following technical lemma. Recall that we consider $U_0$ as a homogeneous space of $H=\mathcal{S}\times \SO^+(1,n)$ with action $(\sigma,h)\cdot g=\sigma g h^{-1}$, $(\sigma,h)\in H$, $g\in U_0$. By \autoref{lemma:partialActionHomSpace}, the stabilizer of $e\in U_0$ in $H$ is then given by $K_\Delta=\{(k,k):k\in K\}$. Let $\Ad(k)g:=kgk^{-1}$ denote the adjoint action of $K$ on $U_0$. This corresponds precisely to the action of $K_\Delta$ on $U_0$.
		\begin{lemma}\label{lemma:approximationMollifier}
			Let $A\subset U_0$ be compact. There exists a compact neighborhood $B$ of $A$ and a sequence of $\Ad(K)$-invariant functions $\chi_j\in C^\infty_c(U_0)$ with support shrinking to $\{\Id\}$ such that the following holds: For every $\psi\in C^\infty_A(U_0)$ there exists $\tilde{\psi}\in C^\infty_c(H)$ such that
			\begin{align*}
				\psi=\lim_{j\rightarrow\infty}\int_H\tilde{\psi}(g)[\pi(g)\chi_j ]d\lambda_H(g)
			\end{align*}
			converges in $C^\infty_{B}(U_0)$. Here, the integral is understood as the Riemann integral of a $C^\infty_B(U_0)$-valued function on $H$ with compact support.
		\end{lemma}
		\begin{proof}
			Let $\tilde{\pi}:H\rightarrow U_0$ be given by $\tilde{\pi}(\sigma,h)=\sigma\cdot h^{-1}$. Note that \autoref{lemma:partialActionHomSpace} shows that $\tilde{\pi}:H\rightarrow U_0$ is a fiber bundle with compact fibers isomorphic to $K_\Delta$. We denote the fiber integration along $\tilde{\pi}$ by $\tilde{\pi}_*$. Pick a function $\zeta\in C^\infty_c(H)$ such that $\tilde{\psi}:=(\psi\circ \tilde{\pi})\zeta\in C^\infty_c(H)$ satisfies $\tilde{\pi}_*\tilde{\psi}=\psi$, which is possible since the fibers are compact. Pick a smooth approximation of the identity $(\eta_j)_j$ in $C^\infty_c(H)$ with supports shrinking to the identity in $H$. By averaging $\eta_j$ over the action of $K_\Delta\times K_\Delta$ given by left and right multiplication (which is possible since $K_\Delta$ is compact), we may assume that $\eta_j$ is both left and right $K_\Delta$-invariant. Then 
			\begin{align*}
				\tilde{\psi}_j:=\int_H \tilde{\psi}(g) [L_g\eta_j]d\lambda_H(g)\in C^\infty_c(H)
			\end{align*}
			converges in $C^\infty_c(H)$ to $\tilde{\psi}$. Here, $L_g$ denotes the action of $g\in H$ on $C^\infty_c(H)$ induced by left multiplication. Note that $\tilde{\pi}:H\rightarrow U_0$ commutes with the action of $H$ on itself and $U_0$. In particular, for $g\in H$, we have $\tilde{\pi}_*(L_g F)=L_g\tilde{\pi}_*F$ for all $F\in C^\infty_c(H)$. This implies in particular that
			\begin{align*}
				\chi_j:=\tilde{\pi}_*\eta_j
			\end{align*}
			is left $K_\Delta$-invariant, which shows that it is $\Ad(K)$-invariant. Note that we may assume that the functions $\eta_j$ are supported on the same compact set in $H$. As the function $\tilde{\psi}$ has compact support, we thus find a compact set $\tilde{B}\subset U_0\times H$ such that $(h,g)\mapsto \tilde{\psi}(g) [L_g\eta_j](h)$ is supported on $\tilde{B}$ for every $j\in\mathbb{N}$. Since the fiber integration $\tilde{\pi}_*:C^\infty_c(H)\rightarrow C^\infty_c(U_0)$ is continuous, we thus obtain
			\begin{align*}
				\psi=&\tilde{\pi}_*(\tilde{\psi})=\lim_{j\rightarrow\infty}\tilde{\pi}_*	\tilde{\psi}_j=\lim_{j\rightarrow\infty}\int_H \tilde{\psi}(g) \tilde{\pi}_*(L_g\eta_j)d\lambda_H(g)\\
				=&\lim_{j\rightarrow\infty}\int_H \tilde{\psi}(g) L_g[\tilde{\pi}_*\eta_j]d\lambda_H(g)=\lim_{j\rightarrow\infty}\int_H \tilde{\psi}(g) [L_g\chi_j]d\lambda_H(g),
			\end{align*}
			where the functions in the integrand take values in $C^\infty_B(U_0)$ for $B=\tilde{\pi}(\tilde{B})$, i.e. the sequence and integral converge in $C^\infty_B(U_0)$. Now note that the action of $g\in H$ on $C^\infty_c(U_0)$ given by $L_g$ is precisely $\pi(g)$.
		\end{proof}
		
		The previous result will allow us to obtain the relevant vanishing properties for $\Lambda^\mu_k$ from its vanishing properties on $\Ad(K)$-invariant functions. This result is in turn based on the main results of \cite{KnoerrRigidmotioninvariant2026}.		
		\begin{proposition}[\cite{KnoerrRigidmotioninvariant2026}*{Proposition 5.4}]
			\label{proposition:FlatCasesimpleValuationsVanish}
			Let $0\le k\le n-1$ and $\varphi:\calP(\R^n)\rightarrow\R$ be a measurable, translation and $\SO(n)$-invariant, as well as simple valuation that is homogeneous of degree $k$. Then $\varphi=0$.
		\end{proposition}

		\begin{theorem}\label{theorem:SimpleLocallyW_n}
			If $\mu\in \CV(\H^n)$ is $\SO^+(1,n)$-invariant and simple, then $\mu$ belongs to $W^{\loc}_n$.
		\end{theorem}
		\begin{proof}
			We show by induction on $0\le k\le n-1$ that $\mu$ belongs to $W^{\loc}_{k+1}$. For $k=0$, this follows from the fact that $\mu_\psi$ is simple for every $\psi\in C^\infty_c(U_0)$ by \autoref{lemma:propertiesLocalSmoothing} (compare also \autoref{remark:W1Loc}). Thus, assume that $\mu$ belongs to $W^{\loc}_{k}$, where $1\le k\le n-1$. We have to show that $\Lambda^\mu_k(\psi)=0$ for all $\psi\in C^\infty_c(U_0)$, which we do in several steps. Note that $\Lambda^\mu_k(\psi)$ is a $k$-homogeneous, translation invariant, simple, and measurable valuation by \autoref{prop:PropertiesFiltration}.
			\begin{enumerate}
				\item If $\psi$ is $\Ad(K)$-invariant, then $\Lambda^\mu_k(\psi)$ is in addition $\SO(e_0^\perp)$-invariant, compare \autoref{lemma:equivPropertiesLambda}. Since it is simple and homogeneous of degree $k\ne n$, it thus vanishes by \autoref{proposition:FlatCasesimpleValuationsVanish}.
				\item If $(\sigma,h)\in H=\mathcal{S}\times \SO^+(1,n)$, we may write $\sigma=n_xa_L$ for $L\in \GL(e_0^\perp)$ and $x\in e_0^\perp$, compare Eq.~\eqref{eq:DefFactorsStabilizer}. \autoref{lemma:equivPropertiesLambda} and \autoref{proposition:InvarianceDerivativesFractionalLinearPart} show that
				\begin{align*}
					&\Lambda^\mu_k(\pi(\sigma,h)\psi)[P]=\Lambda^\mu_k(R_h\psi^{\sigma^{-1}})[P]=\Lambda^\mu_k(\psi^{a_L^{-1}n_{-x}})[P]\\
					&\quad=\Lambda^\mu_k\left({\left(\psi^{a_L^{-1}}\right)}^{n_{-x}}\right)[P]=\Lambda^\mu_k(\psi)[L^{-1}P]
				\end{align*}
				for all $P\in\calP(e_0^\perp)$ and $\psi\in C^\infty_c(U_0)$. In particular, $\Lambda^\mu_k(\pi(\sigma,h)\chi)=0$ for every $\Ad(K)$-invariant $\chi\in C^\infty_c(U_0)$ by (1).
				\item If $\psi\in C^\infty_c(U_0)$ is arbitrary, we may use \autoref{lemma:approximationMollifier} to find a sequence $\chi_j\in C^\infty_c(U_0)$ of $\Ad(K)$-invariant functions and $\tilde{\psi}\in C^\infty_c(H)$ such that
				\begin{align*}
					\psi=\lim_{j\rightarrow\infty}\int_H \tilde{\psi}(g)[\pi(g)\chi_j]d\lambda_H(g)
				\end{align*}
				where the sequence and Riemann integral converge in $C^\infty_B(U_0)$ for a compact neighborhood $B$ of $\supp\psi$. Since $\Lambda^\mu_k(\cdot)[P]$ is continuous on $C^\infty_B(U_0)$ by \autoref{lemma:LambdaContinuousFunctional}, we obtain
				\begin{align*}
					\Lambda^\mu_k(\psi)[P]=\lim_{j\rightarrow\infty}\int_H \tilde{\psi}(g)\Lambda^\mu_k\left(\pi(g)\chi_j\right)[P]d\lambda_H(g).
				\end{align*}
				Since $\chi_j$ is $\Ad(K)$-invariant, the integrand vanishes identically by (2). 
			\end{enumerate}
			Thus $\Lambda^\mu_k$ vanishes identically, which shows that $\mu$ belongs to $W_{k+1}^{\loc}$.
		\end{proof}
	
	\subsection{Characterization of $\Lambda^\mu_n$}
		Recall that the stabilizer of $e_0$ is given by the semi-direct product $\mathcal{S}\cong\mathcal{N}\rtimes \mathcal{A}$, compare Eq.~\eqref{eq:DefFactorsStabilizer}. Consider the smooth character $\chi_0:\mathcal{S}\rightarrow\R\setminus \{0\}$, $\chi_0(\sigma)=\det L$ for $\sigma=n_x a_L$, $x\in e_0^\perp$, $L\in \GL(e_0^\perp)$. Note that $\chi_0$ is in particular invariant under the stabilizer $K$ of $e_0$. This character encodes precisely the action of the differential of elements of the subgroup $\mathcal{A}$ on the space of Lebesgue measures on the tangent space $e_0^\perp$ to $e_0\in\H^n$. The next result uses this observation to describe the action of $\mathcal{S}$ on the derivative $\Lambda^\mu_n$.
		\begin{corollary}\label{corollary:distribution}
			Let $\mu\in\CV(\H^n)$ belong to $W^{\loc}_n$. Then there exists a distribution $c_\mu:C^\infty_c(U_0)\rightarrow\R$ such that
			\begin{align*}
				\Lambda^\mu_n(\psi)=c_\mu(\psi)\vol_n,
			\end{align*}
			where $\vol_n$ denotes the Lebesgue measure on $e_0^\perp$. Moreover, if $\mu$ is $\SO^+(1,n)$-invariant, then
			\begin{align*}
				c_{\mu}(\pi(\sigma,h)\psi)=\chi_0(\sigma)^{-1}c_\mu(\psi)
			\end{align*}
			for $(\sigma,h)\in H=\mathcal{S}\times \SO^+(1,n)$.
		\end{corollary}
		\begin{proof}
			Note that $\Lambda^\mu_n(\psi)$ is a translation invariant and $n$-homogeneous valuation on $\calP(e_0^\perp)$ by \autoref{prop:PropertiesFiltration}. By Hadwiger's characterization of $n$-homogeneous translation invariant valuations on polytopes (see \cite{SchneiderConvexbodiesBrunn2014}*{Theorem 6.4.3}), it is thus a multiple of the Lebesgue measure, so $\Lambda^\mu_n(\psi)=c_\mu(\psi)\vol_n$ for a constant $c_\mu(\psi)\in\R$. If we fix a polytope $P_0$ with $\vol_n(P_0)=1$, then $c_\mu(\psi)=\Lambda^\mu_n(\psi)[P_0]$, so \autoref{lemma:LambdaContinuousFunctional} shows that $c_\mu$ defines a continuous linear functional on $C^\infty_c(U_0)$ and is thus a distribution.\\
			Now assume that $\mu$ is $\SO^+(1,n)$-invariant. For $(\sigma,h)\in H$, write $\sigma=n_xa_L$ for $x\in e_0^\perp$, $L\in\GL(e_0^\perp)$, compare Eq.~\eqref{eq:DefFactorsStabilizer}. From \autoref{lemma:equivPropertiesLambda} and \autoref{proposition:InvarianceDerivativesFractionalLinearPart} we obtain
			\begin{align*}
				&c_\mu(\pi(\sigma,h)\psi)=c_\mu(R_h\psi^{\sigma^{-1}})=c_\mu(R_h\psi^{a_L^{-1}n_{-x}})=\Lambda^\mu_n\left(R_h{\left(\psi^{a_L^{-1}}\right)}^{n_{-x}}\right)[P_0]\\
				&\quad=\Lambda^\mu_n(\psi^{a_L^{-1}})[P_0]=\Lambda^\mu_n(\psi)[L^{-1}P_0]=c_\mu(\psi)\vol_n(L^{-1}P_0)=\det(L)^{-1}c_\mu(\psi),
			\end{align*}
			which shows the last claim due to the definition of $\chi_0$.
		\end{proof}
		\begin{remark}
			Due to \autoref{remark:orderDistribution}, this distribution is always of order at most $n$. Note also that the equivariance property holds for elements belonging to $W^{\loc}_n$ without any invariance assumption if we restrict the action on $c$ to the stabilizer $\mathcal{S}\subset H$ of $e_0$.
		\end{remark}
		We use $\chi_0$ to define a smooth function $w:U_0\rightarrow\R\setminus\{0\}$ in the following way: Recall that $U_0\cong H/K_{\Delta}$ is a homogeneous space of $H=\mathcal{S}\times \SO^+(1,n)$ by \autoref{lemma:partialActionHomSpace}, where the diffeomorphism is induced by the map $\tilde{\pi}:H\rightarrow U_0$, $(\sigma,h)\mapsto \sigma\cdot h^{-1}$. We may define a smooth map $\tilde{w}:H\rightarrow\R$ by $\tilde{w}(\sigma,h):=\chi_0(\sigma)^{-1}$. Then it is easy to check that $\tilde{w}$ is constant on the fibers of $\tilde{\pi}$, which is a smooth submersion, so $\tilde{w}$ descends to the desired smooth function $w$ on $U_0$. In other words, we have $w(\sigma h^{-1})=\chi_0(\sigma)^{-1}$ for $(\sigma,h)\in H$, which is well defined because $\chi_0$ is $K$-invariant by definition.
		\begin{theorem}\label{theorem:InvariantWnlocCharacterization}
			If $\mu\in \CV(\H^n)$ is $\SO^+(1,n)$-invariant and belongs to $W_n^{\loc}$, then there exists $C\in\R$ such that
			\begin{align*}
				c_\mu(\psi)=C\int_{U_0}\psi(g)w(g)d\lambda_{\PGL(n+1,\R)}(g)
			\end{align*}
			for all $\psi\in C^\infty_c(U_0)$.
		\end{theorem}
		\begin{proof}
			Consider the distribution $S$ on $C^\infty_c(U_0)$ defined by $S(\psi)=c_\mu(w^{-1}\psi)$. Then $S$ is $H$-invariant since
			\begin{align*}
				S(\pi(\sigma,h)\psi)&=c_\mu(w^{-1}\pi(\sigma,h)\psi)=\chi_0(\sigma)^{-1}c_\mu(\pi(\sigma,h)[w^{-1}\psi])\\
				&=\chi_0(\sigma)^{-1}\chi_0(\sigma)c_\mu(w^{-1}\psi)=S(\psi)
			\end{align*}
			by \autoref{corollary:distribution}. Thus $S$ is an $H$-invariant distribution on $U_0\cong H/K_{\Delta}$. Since this is a homogeneous space of $H$, $S$ is thus a multiple of the unique $H$-invariant measure on $U_0$. In particular, since the restriction of the Haar measure of $\PGL(n+1,\R)$ to $U_0$ is also $H$-invariant, compare \autoref{lemma:partialActionHomSpace}, $S$ is a multiple of $\lambda_{\PGL(n+1,\R)}|_{U_0}$. Unraveling the definitions, we obtain the desired representation of $c_{\mu}$.
		\end{proof}
\section{Proofs of the main results}
	
	\begin{proof}[Proof of \autoref{maintheorem:SimpleValuation}]
		Let $\mu\in \CV(\H^n)$ be simple and $\SO^+(1,n)$-invariant. Then $\mu$ belongs to $W^{\loc}_n$ by \autoref{theorem:SimpleLocallyW_n}. Let $(\psi_j)_j$ be a smooth approximation of the identity in $C^\infty_c(U_0)$. For $x\in \Omega(e_0)$, we have for all $\psi\in C^\infty_c(U_0)$ such that $\psi^{\tau_x}\in C^\infty_c(U_0)$,
		\begin{align*}
			D^n_x\mu_\psi=\Lambda^\mu_n(\psi^{\tau_x})=c_\mu(\psi^{\tau_x})\vol_n
		\end{align*}
		by \autoref{lemma:TranslationMaps} and \autoref{corollary:distribution}. Fix $P\in\calP(\Omega(e_0))$ and a relatively compact convex open neighborhood $\Omega\subset\Omega(e_0)$ of $P$. Applying \autoref{corollary:EquivalenceLocalW_k}, we thus obtain a compact neighborhood $A$ of the identity in $U_0$ such that $G_{e_0*}\mu_{\psi}\in W_n(\Omega)$ for all $\psi\in C^\infty_A(U_0)$. Shrinking $A$ if necessary, we may assume that $\psi^{\tau_x}\in C^\infty_c(U_0)$ for all $\psi\in C^\infty_A(U_0)$ and $x\in \Omega$. Since the supports of $\psi_j$ shrink to the identity, we find $M>0$ such that $\supp \psi_j\subset A$ for all $j\ge M$. Now \autoref{prop:PropertiesFiltration} shows that for all $j\ge M$,
		\begin{align*}
			\mu_{\psi_j}(G_{e_0}^{-1}(P))=G_{e_0*}\mu_{\psi_j}(P)=\int_{P}c_\mu(\psi_j^{\tau_x}) dx.
		\end{align*}
		By \autoref{theorem:InvariantWnlocCharacterization}, there is a constant $C\in\R$ independent of $P$ and $\psi_j$ such that
		\begin{align*}
			c_\mu(\psi_j^{\tau_x})=&C\int_{U_0}\psi^{\tau_x}_j(g)w(g)d\lambda_{\PGL(n+1,\R)}(g)\\
			=&C\int_{U_0}\psi_j(g)w(\tau_{-x}g)d\lambda_{\PGL(n+1,\R)}(g),
		\end{align*}
		where we used the invariance of the Haar measure (and that $\tau_{-x}g\in U_0$ for all $g\in \supp\psi_j$ by construction). Since $(\psi_j)_j$ is a smooth approximation of the identity, $c_\mu(\psi_j^{\tau_x})$ converges uniformly in $x\in P$ to $Cw(\tau_{-x})$. In particular, since $\mu_{\psi_j}(G_{e_0}^{-1}(P))$ converges to $\mu(G^{-1}_{e_0}(P))$ by \autoref{lemma:propertiesLocalSmoothing}, we obtain
		\begin{align*}
			\mu(G^{-1}_{e_0}(P))=C\int_{P}w(\tau_{-x})dx.
		\end{align*}
		Since this holds for all $P\in\calP(\Omega(e_0))$, $G_{e_0*}\mu$ and (therefore) $\mu$ are smooth measures on $\calP(\Omega(e_0))$ and $\calP(\H^n)$ respectively. As $\mu$ is $\SO^+(1,n)$-invariant by assumption, it has to be a multiple of the hyperbolic volume.
	\end{proof}
	
	\autoref{maintheorem:hyperbolicHadwiger} now follows with a standard argument from \autoref{maintheorem:SimpleValuation}, which we include for completeness.	
	\begin{proof}[Proof of \autoref{maintheorem:hyperbolicHadwiger}]
		We use induction on the dimension $n$ of $\H^n$, where the case $n=0$ is vacuously true since $\calP(\H^0)$ only contains a singleton. Assume that the claim holds for $\H^{n-1}$. If $\mu\in\CV(\H^n)$ is $\SO^+(1,n)$-invariant, consider its restriction to $\H^n\cap \{u\in \R^{n+1}:u_n=0\}\cong \H^{n-1}$. Then $\mu|_{\H^{n-1}}\in\CV(\H^{n-1})$ is $\SO^+(n-1,1)$-invariant and thus a linear combination of the hyperbolic intrinsic volumes by the induction hypothesis. In particular, we find $c_0,\dots,c_{n-1}\in\R$ such that $\tilde{\mu}:=\mu-\sum_{j=0}^{n-1}c_jV^h_j$ vanishes on polytopes in $\H^{n-1}$. Since $\tilde{\mu}$ is $\SO^+(1,n)$-invariant by construction, it thus vanishes on all lower dimensional polytopes in $\calP(\H^n)$, so $\tilde{\mu}$ is a simple valuation. By \autoref{maintheorem:SimpleValuation}, $\tilde{\mu}=c_nV^h_n$ is a multiple of the hyperbolic volume, and we obtain
		\begin{align*}
			\mu=\tilde{\mu}+\sum_{j=0}^{n-1}c_jV^h_j=\sum_{j=0}^{n}c_jV^h_j,
		\end{align*}
		which completes the proof.
	\end{proof}

\bibliographystyle{plain}
\bibliography{../../library/library.bib}

\Addresses
	
\end{document}